\documentclass[12pt]{amsart}
\usepackage{amsmath,amssymb,amsthm,amscd}
\usepackage[latin1]{inputenc}
\usepackage{wasysym}
\usepackage[all]{xy}
\usepackage[bf,small]{caption}
\usepackage{subfigure}
\usepackage{cite}
\usepackage{paralist}
\usepackage{graphicx}
\usepackage[usenames,dvipsnames]{color}
\usepackage{eurosym}
\usepackage{setspace}
\usepackage{dsfont}
\usepackage{fancyhdr}
\usepackage{mathdots}
\usepackage{remreset}
\usepackage[inner=2.5cm,outer=2.5cm,top=2cm,bottom=2cm,
includeheadfoot]{geometry}
\newtheorem{theorem}{Theorem}[section]

\newtheorem{lemma}[theorem]{Lemma}
\newtheorem{prop}[theorem]{Proposition}
\newtheorem{defi}[theorem]{Definition}
\newtheorem{rem}[theorem]{Remark}
\newtheorem*{theo}{Theorem}

\DeclareMathOperator \im{im}
\DeclareMathOperator \CR{CR}

\DeclareMathOperator \D{D}
\DeclareMathOperator \id{id}

\DeclareMathOperator \modi{mod}

\DeclareMathOperator \Hom{Hom}
\DeclareMathOperator \End{End}
\DeclareMathOperator \EIP{EIP}
\DeclareMathOperator \EKP{EKP}
\DeclareMathOperator \CJT{CJT}

\DeclareMathOperator \rad{rad}

\DeclareMathOperator \rk{rk}
\DeclareMathOperator \op{op}

\definecolor{mygray}{gray}{.50}

\begin{document}
\title{AR-components for generalized Beilinson algebras}
\author{Julia Worch}
\address{Christian-Albrechts-Universit\"at zu Kiel, Ludewig-Meyn-Str. 4, 24098 Kiel, Germany}
\email{worch@math.uni-kiel.de}

\thanks{Partly supported by the D.F.G. priority program SPP 1388  ``Darstellungstheorie''}

\begin{abstract}
We show that the generalized $W$-modules defined in \cite{w12} determine $\mathbb{Z}A_{\infty}$-components in the Auslander-Reiten quiver $\Gamma(n,r)$ of the generalized Beilinson algebra $B(n,r)$, $n \geq 3$. These components entirely consist of modules with the constant Jordan type property. We arrive at this result by interpreting $B(n,r)$ as an iterated one-point extension of the $r$-Kronecker algebra $\mathcal{K}_r$ which enables us to generalize findings concerning the Auslander-Reiten quiver $\Gamma(\mathcal{K}_r)$ presented in \cite{w12} to $\Gamma(n,r)$.
\end{abstract}

\maketitle

\section*{Introduction}

Motivated by work of Carlson, Friedlander, Pevtsova and Suslin on modular representations of elementary abelian $p$-groups (cf. \cite{cafrpe08}, \cite{cfs09}), we introduced in a foregoing paper \cite{w12} modules with the equal images property, the equal kernels property and modules of constant Jordan type for the generalized Beilinson algebra $B(n,r)$, the path algebra of the quiver
$$\begin{xy}
\xymatrix{
1 \ar@/^1pc/[r]^{\gamma_1} \ar@/_1pc/[r]_ {\gamma_r} \ar@{}[r]|{\vdots} & 2 \ar@/^1pc/[r]^{\gamma_1} \ar@/_1pc/[r]_ {\gamma_r}  \ar@{}[r]|{\vdots} & 3 \ar@{}[r]|{ \cdots } & n-1 \ar@/^1pc/[r]^{\gamma_{1}} \ar@/_1pc/[r]_ {\gamma_{r}}  \ar@{}[r]|{\vdots} & n }
\end{xy}$$
modulo commutativity relations $\gamma_i\gamma_j=\gamma_j\gamma_i$. These classes are defined such that a faithful exact functor $\mathfrak{F} \colon \modi B(n,r) \rightarrow \modi kE_r$ maps a $B(n,r)$-module with one of the above properties to a module that satisfies the respective property over an elementary abelian $p$-group $E_r$ of rank $r \geq 2$ \cite[2.3]{w12}. We moreover gave a generalization of the so-called $W$-modules, a special class of $kE_2$-modules with the equal images property defined in \cite{cfs09} via generators and relations, to elementary abelian $p$-groups of arbitrary rank and showed that these modules might as well be considered modules with the equal images property over $B(n,r)$ via $\mathfrak{F}$.\\

The Auslander-Reiten quiver constitutes an important invariant of the Morita equivalence class of an algebra. While in general it is hard to compute the components of this quiver, one has a good knowledge of the Auslander-Reiten components for group algebras of finite groups and hereditary algebras. In particular, by work of Ringel \cite{ri78} and Erdmann \cite{erd90}, the regular Auslander-Reiten components of wild hereditary algebras and of $p$-elementary abelian groups are of tree class $A_{\infty}$. On the contrary, not much is known about the Auslander-Reiten theory of generalized Beilinson algebras. These algebras, however, constitute an interesting class of algebras of global dimension $\geq 2$ for $n \geq 3$. \cite[3.7]{w13}.\\
In the present paper, we interpret $B(n,r)$ as an iterated one-point extension of the hereditary path algebra $B(2,r)$ of the $r$-Kronecker by duals of generalized $W$-modules. As a consequence, we can make use of certain lifting properties of Auslander-Reiten sequences \cite[2.5]{ri84} in combination with torsion theoretic arguments to generalize our findings from \cite{w12} as follows:

\begin{theo} Let $r \geq 2$, $m>n \geq 2$. 
\begin{enumerate}[(i)]
\item If $(n,r) \neq (2,2)$, then the generalized $W$-module $W_{m,n}^{(r)}$ is a quasi-simple module in a $\mathbb{Z}A_{\infty}$-component $\mathcal{C}_m$ of $\Gamma(n,r)$ which contains two disjoint cones: one consisting of modules with the equal images property and one consisting of modules with the equal kernels property. Moreover, all modules in $\mathcal{C}_m$ have constant Jordan type.
\item If $r>2$, then the two cones in (i) are adjacent. In case $r=2$, there is exactly one quasi-simple module that neither has the equal images nor the equal kernels property.
\end{enumerate}
\end{theo}
For $r>2$, these components can be visualized as follows with $W=W_{m,n}^{(r)}$ being the uniquely determined generalized $W$-module of the component.
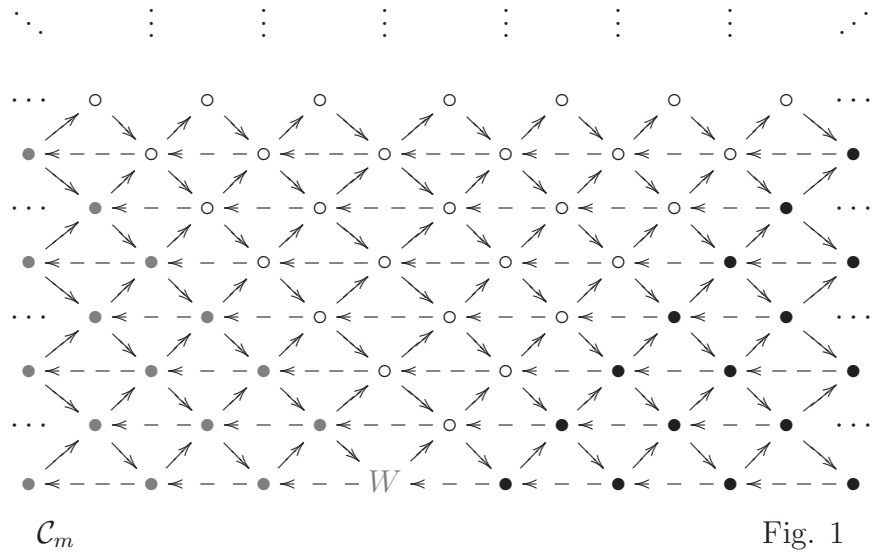
\[ \begin{xy}
\xymatrix@R=0.7em@C=0.7em{
\ddots  && \vdots && \vdots &&\vdots  && \vdots && \vdots &&\vdots && \iddots\\
\cdots &   \circ   \ar[rd]  &&  \circ   \ar[rd]   && \circ   \ar[rd]   &&  \circ  \ar[rd] &&  \circ   \ar[rd]  &&  \circ   \ar[rd] && \circ  \ar[rd]  & \cdots \\
 {\color{mygray}\bullet} \ar[rd] \ar[ru] && \circ  \ar[rd] \ar[ru] \ar@{-->}[ll] && \circ  \ar[rd] \ar[ru] \ar@{-->}[ll] && \circ  \ar[rd] \ar[ru] \ar@{-->}[ll] && \circ  \ar[rd] \ar[ru] \ar@{-->}[ll] && \circ  \ar[rd] \ar[ru] \ar@{-->}[ll] && \circ \ar[rd] \ar[ru]  \ar@{-->}[ll] &&{\bullet}  \ar@{-->}[ll] \\
  \cdots & {\color{mygray}\bullet} \ar[ru]  \ar[rd]  &&  \circ  \ar[ru]  \ar[rd] \ar@{-->}[ll]  && \circ  \ar[ru]  \ar[rd] \ar@{-->}[ll]  &&  \circ  \ar[ru]  \ar[rd] \ar@{-->}[ll] &&  \circ \ar[ru]  \ar[rd] \ar@{-->}[ll]  && \circ  \ar[ru]  \ar[rd] \ar@{-->}[ll]   && {\bullet}   \ar[rd] \ar@{-->}[ll] \ar[ru] & \cdots \\
{\color{mygray}\bullet}\ar[ru]  \ar[rd] &&   {\color{mygray}\bullet} \ar[ru]  \ar[rd] \ar@{-->}[ll] &&  \circ   \ar[ru]  \ar[rd] \ar@{-->}[ll]  &&  \circ  \ar[ru]  \ar[rd] \ar@{-->}[ll] &&  \circ   \ar[ru]  \ar[rd] \ar@{-->}[ll] &&  \circ   \ar[ru]  \ar[rd] \ar@{-->}[ll]&&  {\bullet}  \ar[rd] \ar@{-->}[ll] \ar[ru] && {\bullet}  \ar@{-->}[ll] \\
 \cdots & {\color{mygray}\bullet} \ar[ru]  \ar[rd]  &&   {\color{mygray}\bullet} \ar[ru]  \ar[rd] \ar@{-->}[ll]  &&  \circ  \ar[ru]  \ar[rd] \ar@{-->}[ll]  &&  \circ  \ar[ru]  \ar[rd] \ar@{-->}[ll] &&  \circ  \ar[ru]  \ar[rd] \ar@{-->}[ll]   && {\bullet}   \ar[ru]  \ar[rd] \ar@{-->}[ll] && {\bullet}  \ar[ru]  \ar[rd] \ar@{-->}[ll] & { \cdots} \\
{\color{mygray}\bullet} \ar[rd]  \ar[ru] && {\color{mygray}\bullet} \ar[rd]  \ar[ru] \ar@{-->}[ll] &&  {\color{mygray}\bullet} \ar[rd]  \ar[ru] \ar@{-->}[ll] && \circ  \ar[rd]  \ar[ru] \ar@{-->}[ll] && \circ  \ar[rd]  \ar[ru] \ar@{-->}[ll] &&  {\bullet} \ar[rd]  \ar[ru] \ar@{-->}[ll] && {\bullet} \ar[ru] \ar[rd] \ar@{-->}[ll] && {\bullet}  \ar@{-->}[ll] \\
 \cdots & {\color{mygray}\bullet} \ar[ru]  \ar[rd]  &&  {\color{mygray}\bullet} \ar[ru]  \ar[rd] \ar@{-->}[ll]  &&   {\color{mygray}\bullet} \ar[ru]  \ar[rd] \ar@{-->}[ll]  &&  \circ  \ar[ru]  \ar[rd] \ar@{-->}[ll] && {\bullet} \ar[ru]  \ar[rd] \ar@{-->}[ll]   && {\bullet}  \ar[ru]  \ar[rd] \ar@{-->}[ll] && {\bullet} \ar[rd] \ar[ru]  \ar@{-->}[ll] & { \cdots} \\
{\color{mygray}\bullet}  \ar[ru] &&  {\color{mygray}\bullet}   \ar[ru] \ar@{-->}[ll] &&  {\color{mygray}\bullet}   \ar[ru] \ar@{-->}[ll] &&   {\color{mygray} W}  \ar[ru] \ar@{-->}[ll] && {\bullet} \ar[ru] \ar@{-->}[ll]&&   {\bullet} \ar[ru] \ar@{-->}[ll] && {\bullet} \ar[ru] \ar@{-->}[ll]&& {\bullet} \ar@{-->}[ll]}
\end{xy}  
\]
\[\mathcal{C}_m \hspace{9 cm} \text{ Fig. 1} \]
The gray and black bullets refer to modules with the equal images property and the equal kernels property, respectively.\\

Throughout, we denote by $k$ an algebraically closed field and we let $n,r \geq 2$. We assume the reader to be familiar with the concept of an algebra given by a quiver with relations and refer to \cite{ass06} and \cite{ars} for basic notions of Auslander-Reiten theory.

\section{Module categories for $B(n,r)$ and the case $n=2$}\label{two}
In this section, we recall the module categories we defined in \cite{w12} and provide the reader with the relevant information we have on the Auslander-Reiten quiver of $B(2,r)$.\\

Let $E(n,r)$ be the path algebra of the quiver $\mathcal{Q}(n,r)$ with $n$ vertices and $r$ arrows between the vertices $i$ and $i+1$ for all $0 \leq i \leq n-2$:
$$\begin{xy}
\xymatrix{
0 \ar@/^1pc/[r]^{\gamma_1^{(0)}} \ar@/_1pc/[r]_ {\gamma_r^{(0)}} \ar@{}[r]|{\vdots} & 1 \ar@/^1pc/[r]^{\gamma_1^{(1)}} \ar@/_1pc/[r]_ {\gamma_r^{(1)}}  \ar@{}[r]|{\vdots} & 2 \ar@{}[r]|{\cdots} & n-2 \ar@/^1pc/[r]^{\gamma_{1}^{({n-2})}} \ar@/_1pc/[r]_ {\gamma_{r}^{({n-2})}}  \ar@{}[r]|{\vdots} & n-1 }
\end{xy}$$
The {\bf generalized Beilinson algebra $B(n,r)$} is the factor algebra $E(n,r)/\mathcal{I}$ where $\mathcal{I}$ is the ideal generated by the commutativity relations $\gamma^{(i+1)}_s \gamma^{(i)}_{t}-\gamma^{(i+1)}_{t}  \gamma^{(i)}_{s}$ for all $s,t \in \left\{1,\ldots,r\right\}$ and $ i \in \left\{0,\ldots,n-1\right\}$.
These algebras generalize the algebras of the form $B(n)=B(n,n)$ introduced by Beilinson in \cite{bei78}.
We denote by $S(i)$, $P(i)$, $I(i)$ and $e_i$ the simple, the projective and the injective indecomposable $B(n,r)$-module and the primitive orthogonal idempotent corresponding to the vertex $i \in \left\{0,\ldots,n-1\right\}$. For $\alpha \in k^r \backslash 0$ and $M \in \modi B(n,r)$, we consider the linear operator
$\alpha_M \colon M \rightarrow M$ given by left-multiplication with the element $\sum_{i=0}^{n-2}(\alpha_1 \gamma_1^{(i)} + \cdots + \alpha_r \gamma_r^{(i)}) \in B(n,r)$.

\begin{defi}[cf. \cite{w12}]\label{defeip} We denote by
\begin{enumerate}[(i)]
\item $\EIP(n,r) := \left\{M \in \modi B(n,r) \mid \forall  \alpha \in k^r\backslash 0:  \im(\alpha_M)=\bigoplus_{i=1}^{n-1}e_i.M  \right\}$
\item $\EKP(n,r) := \left\{M \in \modi B(n,r) \mid  \forall  \alpha \in k^r\backslash 0: \ker(\alpha_M)=e_{n-1}.M  \right\}$
\item $\CR^j(n,r) := \left\{ M \in \modi B(n,r) \mid \exists c_j \in \mathbb{N}_0 \ \forall  \alpha \in k^r\backslash 0:  \rk (\alpha_{M})^j=c_j  \right\}$
\item $\CJT(n,r):= \bigcap_{j=1}^n  \CR^j(n,r)$.
\end{enumerate}
the categories of modules with the equal images property, the equal kernels property and the categories of modules of constant $j$-rank and of constant Jordan type, respectively.
\end{defi}
Due to the fact that $B(n,r) \cong B(n,r)^{\op}$, there is a duality $\D \colon \modi B(n,r) \rightarrow \modi B(n,r)$ induced by taking the linear dual and relabelling the vertices of the quiver $\mathcal{Q}(n,r)$ in the reversed order. The functor $\D$ is compatible with the Auslander-Reiten translation $\tau$ in the sense that $\D \tau \cong \tau^{-1} \D$ \cite[3.3]{w13} and restricts to a duality between the categories $\EIP(n,r)$ and $\EKP(n,r)$. Note, moreover, that $\EIP(n,r) \cup \EKP(n,r) \subseteq \CJT(n,r)$. The following is a direct consequence of Definition \ref{defeip}.
\begin{rem}\label{equalzero} Let $M \in \EIP(n,r)$, $N \in \EKP(n,r)$.
\begin{enumerate}[(i)]
\item If $e_0.M=0$, then $M=0$.
\item If $e_{n-1}.N=0$, then $N=0$.
\end{enumerate}
\end{rem}

A homological characterization of these categories \cite[2.2.1]{w12} yields that the category $\EIP(n,r)$ is the torsion class $\mathcal{T}$ of a torsion pair $(\mathcal{T}, \mathcal{F})$ in $\modi B(n,r)$ with $\EKP(n,r) \subset \mathcal{F}$ that is closed under the Auslander-Reiten translation $\tau$ and contains all preinjective modules\cite[2.2.3]{w12}. In particular, there are no non-trivial maps $\EIP(n,r) \rightarrow \EKP(n,r)$. This yields \cite[2.9]{w12}:
\begin{theorem}\label{ZA}
Let $\mathcal{C}$ be a regular $\mathbb{Z}A_{\infty}$-component of $\Gamma(n,r)$. If $\EIP(n,r) \cap \mathcal{C} \neq \emptyset$, then either $\mathcal{C} \subseteq \EIP(n,r)$ or there exists a quasi-simple module $W_{\mathcal{C}}$ such that
 $(\rightarrow W_{\mathcal{C}})=\mathcal{C} \cap \EIP(n,r).$
Dually, if $\EKP(n,r) \cap \mathcal{C} \neq \emptyset$, then either $\mathcal{C} \subseteq \EKP(n,r)$ or there exists a quasi-simple module $M_{\mathcal{C}}$ such that
$(M_{\mathcal{C}} \rightarrow)=\mathcal{C} \cap \EKP(n,r).$
\end{theorem}
Here, for $M \in \modi B(n,r)$ indecomposable, we denote by $(M \rightarrow)$ and $(\rightarrow M)$ the sets consisting of $M$ and all successors and all predecessors of $M$ in $\Gamma(n,r)$, respectively. In case $\mathcal{C}$ contains both, modules with the equal images and the equal kernels property, we define
$\mathcal{W}(\mathcal{C})$ by the property $$\tau^{\mathcal{W}(\mathcal{C})+1}M_{\mathcal{C}}=W_{\mathcal{C}}$$ i.e. $\mathcal{W}(\mathcal{C})$ is the number of quasi-simple modules in $\mathcal{C}$ that neither satisfy the equal images nor the equal kernels property.\\
We furthermore denote by $R=k[X_1,\ldots,X_r]$ the polynomial ring in $r$ variables and by $I=(X_1,\ldots,X_r)$ the ideal generated by $X_1,\ldots,X_r$.
There is an equivalence of categories $\modi B(n,r) \cong \mathcal{C}_{[0,n-1]}$, where $\mathcal{C}_{[0,n-1]}$ denotes the category of $\mathbb{Z}$-graded $R$-modules with support contained in $\left\{0,\ldots,n-1 \right\}$. We denote by $[-]$ the shift to the right in the category of $\mathbb{Z}$-graded $R$-modules. Via this identification the $\mathbb{Z}$-graded $R$-module
$$M_{m,n}^{(r)}=(I^{m-n}/I^n)[n-m],$$
$m \geq n$, is an indecomposable object in $\EKP(n,r)$ with $W_{m,n}^{(r)}:=\D M_{m,n}^{(r)} \in \EIP(n,r)$ \cite[3.6]{w13}. In case $m<n$, we define $M_{m,n}^{(r)}:=M_{n,n}^{(r)}$. We call modules of the form $W_{m,n}^{(r)}$ {\bf generalized $W$-modules}, since for $n \leq p$ and $r=2$, these modules correspond to the $kE_2$-modules defined by Carlson, Friedlander and Suslin in \cite{cfs09} via generators and relations. The module ${M_{3,2}^{(3)}} \in \EKP(2,3)$, for example, can be depicted as follows:
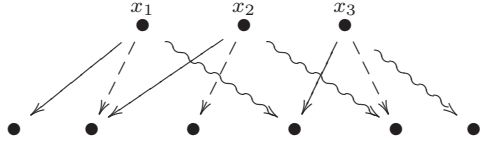
\[ \begin{xy}
\xymatrix@R=2em@C=0.4em{
& & & \stackrel{x_1}{\bullet} \ar@{->}[llld] \ar@{-->}[ld] \ar@{~>}[rrrd] &&  \stackrel{x_2}{\bullet} \ar@{->}[llld] \ar@{-->}[ld] \ar@{~>}[rrrd]&&  \stackrel{x_3}{\bullet} \ar@{~>}[rrrd] \ar@{-->}[rd] \ar@{->}[ld]\\
{\bullet} &&  {\bullet}  && \bullet  && \bullet  && \bullet  && \bullet}
\end{xy} \]
The dots represent the canonical basis elements given by the monomials in degree one and two and $\rightarrow, \ \dashrightarrow$ and $\rightsquigarrow$ denote the action of $\gamma_1^{(0)}, \gamma_2^{(0)}$ and $\gamma_3^{(0)}$, respectively.\\

The algebra $B(2,r)$ is the path algebra $\mathcal{K}_r$ of the $r$-Kronecker quiver. Whenever $r>2$, $B(2,r)$ is wild and due to a result by Ringel \cite{ri78} all regular components are of type $\mathbb{Z}A_{\infty}$. In \cite[\S 3]{w12}, we have shown that the module $W_{m,2}^{(r)}$ is quasi-simple in a $\mathbb{Z}A_{\infty}$-component $\mathcal{C}_m$ of $\Gamma(2,r)$ with $W_{\mathcal{C}_m}=W_{m,2}^{(r)}$ and $\mathcal{W}(\mathcal{C}_m)=0$. Thus the module $W_{m,2}^{(r)}$ is in the rightmost position in the equal images cone of $\mathcal{C}_m \subseteq \Gamma(2,r)$, i.e. $\mathcal{C}_m$ can be visualized as in Fig. 1 with $W=W_{m,2}^{(r)}$. The dual statement holds for modules of the form $M_{m,2}^{(r)}$.

\section{One-point extensions}
We will now provide the necessary theoretical framework and notation for the theory of one-point extensions. For a general introduction, the reader is referred to \cite{ri84} or \cite[XV.1]{ass07b}.
\begin{defi}[Ringel \cite{ri84}]\label{ext}
Let $A$ be an algebra, $M$ in $\modi A$. The algebra
$$A[M]={\begin{pmatrix} A & M \\ 0 & k \end{pmatrix}}$$
with usual matrix addition and multiplication is referred to as the {\bf one-point extension} of $A$ by $M$.
\end{defi}
If $A=kQ_A/I$ is a basic algebra, we obtain the quiver $Q_{A[M]}$ of $A[M]$ by adding a source vertex together with some arrows to $Q_A$.
A module over $A[M]$ is of the form ${\begin{pmatrix} N \\ V \end{pmatrix}}_{\varphi},$ where $N \in \modi A$, $V \in \modi k$ and $\varphi \in \Hom_k (V, \Hom_A(M,N))$. The $A[M]$-module structure is then given via
$${\begin{pmatrix} a & m \\ 0 & \lambda \end{pmatrix}} {\begin{pmatrix} n \\ v \end{pmatrix}} = {\begin{pmatrix} a.n + \varphi(v)(m) \\ \lambda v \end{pmatrix}}.$$
Given $\tilde{X}={\begin{pmatrix} X \\ V \end{pmatrix}}_{\varphi},\; \tilde{Y}={\begin{pmatrix} Y \\ W \end{pmatrix}}_{\psi} \in \modi A[M]$, a morphism $\tilde{X} \rightarrow \tilde{Y}$ in $\modi A[M]$ corresponds to a pair $(f_0,f_1)$, where $f_0 \in \Hom_A(X,Y), f_1 \in \Hom_k(V,W)$ and $\Hom_A(M,f_0) \circ \varphi = \psi \circ f_1$.
Since $A$ is a factor algebra of $A[M]$, there is a full exact embedding $\iota_A: \modi A \rightarrow \modi A[M]$, sending $N \in \modi A$ to the $A[M]$-module ${\begin{pmatrix} N \\ 0 \end{pmatrix}}_{0}$. \\
There is a simple injective module $\tilde{S}={\begin{pmatrix} 0 \\ k \end{pmatrix}}_0 \in \modi A[M]$. The indecomposable projective $A[M]$-modules are exactly the images of the projective indecomposables of $A$ under $\iota_A$ together with the module
$$P(\tilde{S})={\begin{pmatrix} M \\ k \end{pmatrix}}_{\lambda \mapsto \lambda \id_M}.$$

The following lemma due to Ringel gives information on how almost split sequences in $\modi A$ ``lift'' to $\modi A[M]$ \cite[2.5]{ri84}.
\begin{lemma}\label{lift}
Let $A$ be an algebra, $M$ an $A$-module. Let furthermore
$$0 \rightarrow \tau N  \stackrel{f}{\rightarrow} E  \stackrel{g}{\rightarrow} N \rightarrow 0$$ be an Auslander-Reiten sequence in $\modi A$. Then
$$0 \rightarrow {\begin{pmatrix} \tau N \\ \Hom_A(M,\tau N)\end{pmatrix}}_{\id} \stackrel{{\begin{pmatrix} f \\ \id \end{pmatrix}}}{\rightarrow} {\begin{pmatrix} E \\ \Hom_A(M,\tau N)\end{pmatrix}}_{f^*} \stackrel{{\begin{pmatrix} g \\ 0 \end{pmatrix}}}{\rightarrow} {\begin{pmatrix} N \\ 0 \end{pmatrix}}_0 \rightarrow 0$$
is an Auslander-Reiten sequence in $\modi A[M]$.
\end{lemma}
\section{Beilinson algebras as iterated one-point extensions}

The simple module $S(0) \in \modi B(n,r)$ is injective. In view of \cite[XV.1]{ass07b}, we hence obtain 
$$B(n,r) \cong  (1-e_0)B(n,r)(1-e_0)[\rad P(0)].$$

Since $B(n-1,r)$ is isomorphic to the algebra $(1-e_0)B(n,r)(1-e_0)$, we can thus regard
$B(n,r)$ as a one-point extension of $B(n-1,r)$ and consider $\rad P(0)$ an object in $\modi B(n-1,r)$. By \cite[3.23]{w13}, we have isomorphisms $M_{n,n}^{(r)} \cong P(0)$ in $\modi B(n,r)$ and $\rad P(0) \cong M_{n,n-1}^{(r)} $ in $\modi B(n-1,r)$. Inductively, we obtain
$$
B(n,r) \cong B(n-1,r) [M_{n,n-1}^{(r)}] \cong \mathcal{K}_r[M_{3,2}^{(r)}] \cdots [M_{n,n-1}^{(r)}].
$$

\bigskip

For $n=3$, $r=2$, we can visualize this as follows: Extending the path algebra $\mathcal{K}_2$ of the quiver
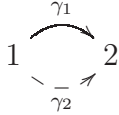
\[ \begin{xy} \xymatrix{
1 \ar@/^1pc/[r]^{\gamma_1} \ar@{-->}@/_1pc/[r]_{\gamma_2}  & 2} \end{xy} \] by the module $M_{3,2}^{(2)}$
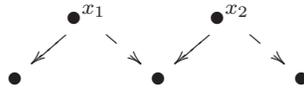
\[ \begin{xy}
\xymatrix@R=0.8em@C=0.8em{
& \bullet^{x_1} \ar[ld] \ar@{-->}[rd] && \bullet^{x_2} \ar[ld] \ar@{-->}[rd]\\
\bullet && \bullet && \bullet}
\end{xy} \]
yields the path algebra of 
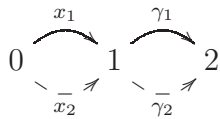
\[ \begin{xy}\xymatrix{
0 \ar@/^1pc/[r]^{x_1} \ar@{-->}@/_1pc/[r]_ {x_2} & 1 \ar@/^1pc/[r]^{\gamma_1} \ar@{-->}@/_1pc/[r]_{\gamma_2}  & 2} \end{xy} \] with relations $\gamma_2.x_1=\gamma_1.x_2$, which is easily seen to be isomorphic to $B(3,2)$.\\

From now on, we will identify the algebras $B(n,r)$ and $B(n-1,r)[M_{n,n-1}^{(r)}]$. Note that when writing $\tilde{M} \in \modi B(n,r)$ in the form $\tilde{M}={\begin{pmatrix} M \\ V \end{pmatrix}}_\varphi \in \modi B(n-1,r) [M_{n,n-1}^{(r)}]$, the dimension vector $\underline{\dim} \tilde{M}$ coincides with the vector $(\dim_k V, \underline{\dim} M)$.\\
We want to study the Auslander-Reiten quiver $\Gamma(n,r)$ by making use of the information on $\Gamma(2,r)$ presented in Section 2. An application of Lemma \ref{lift} yields \cite[5.6]{w13}:

\begin{prop}\label{wmcomps}
Let $m \geq n\geq 3$. We have
\begin{enumerate}
\item[(i)] $\iota_{B(n-1,r)} \tau^{-1}_{B(n-1,r)} M_{m,n-1}^{(r)} \cong \tau^{-1}_{B(n,r)}M_{m,n}^{(r)}$,
\item[(ii)]  $\iota_{B(n-1,r)} \tau^{-1}_{B(n-1,r)} W_{m,n-1}^{(r)} \cong \tau^{-1}_{B(n,r)}W_{m+1,n}^{(r)}$.
\end{enumerate}
\end{prop}
Hence the Auslander-Reiten sequences in $\modi B(n-1,r)$ starting in $M_{m,n-1}^{(r)}$ and $W_{m,n-1}^{(r)}$ lift to Auslander-Reiten sequences in $\modi B(n,r)$ that start in $M_{m,n}^{(r)}$ and $W_{m+1,n}^{(r)}$, respectively.

\section{Occurrence of generalized $W$-modules in $\Gamma(n,r)$}
We now show that generalized $W$-modules determine $\mathbb{Z}A_{\infty}$-components in $\Gamma(n,r)$, $n \geq 3$, $r \geq 2$, that entirely consist of modules with the constant Jordan type property.\\

Let us consider the case $r=2$. On the level of the Auslander-Reiten quiver $\Gamma (2,2)$ of the tame algebra $\mathcal{K}_2$, we do not have any $\mathbb{Z}A_{\infty}$-components to start out with and due to \cite[4.2.2]{f11}, the $W$-modules correspond to the preinjective $\mathcal{K}_2$-modules. With the use of tilting theory one can show that all regular components of $\Gamma(3,2)$ are of type $\mathbb{Z}A_{\infty}$ as has been communicated to me by Otto Kerner \cite{ker13}: There exists a preprojective tilting module $T$ over the path algebra of the extended Kronecker quiver

\[ \begin{xy}\xymatrix{
0 \ar@{->}[r] & 1 \ar@/^1pc/[r] \ar@{->}@/_1pc/[r] & 2} \end{xy} \]

such that $\End(T)$ is isomorphic to $B(3,2)$. The regular components of $\Gamma(\End(T))$ are of type $\mathbb{Z}A_{\infty}$, while there is a preprojective and preinjective component consisting of the $\tau$- and $\tau^{-1}$-shifts of all projective indecomposables and injective indecomposables, respectively.

\begin{prop}\label{dreizwei}
For $m > 3$, the module $W_{m,3}^{(2)}$ is quasi-simple in a $\mathbb{Z}A_{\infty}$-component $\mathcal{C}_m$ of $\Gamma(3,2)$ with $\mathcal{W}(\mathcal{C}_m)=1$. Dually, $M_{m,3}^{(2)}$ is quasi-simple in a $\mathbb{Z}A_{\infty}$-component of $\Gamma(3,2)$.
\end{prop}
\begin{proof}
Consider $W_{m-1,2}^{(2)} \in \modi \mathcal{K}_2$. Due to Proposition \ref{wmcomps} $(ii)$, we have an isomorphism $\tau_{B(3,2)}^{-1}W_{m,3}^{(2)} \cong \iota_{B(2,2)}\tau^{-1}_{\mathcal{K}_2} W_{m-1,2}^{(2)}$. Note that for $k \geq 1$, the module $W^{(2)}_{k,2}$ is the preinjective $\mathcal{K}_2$-module with dimension vector $(k-1,k)$ \cite[4.2.2]{f11}. It is well-known that there is an Auslander-Reiten sequence

$$ 0 \rightarrow W_{m-1,2}^{(2)} \rightarrow  W_{m-2,2}^{(2)} \oplus W_{m-2,2}^{(2)} \rightarrow  W_{m-3,2}^{(2)} \rightarrow 0 $$
in $\modi \mathcal{K}_2$. We hence obtain
$$\tau_{B(3,2)}^{-1} W_{m,3}^{(2)} \cong \iota_{B(2,2)} W_{m-3,2}^{(2)}.$$
Remark \ref{equalzero} yields that $\iota_{B(2,2)} W_{m-3,2}^{(2)} \notin \EIP(3,2)$ since $e_0.\iota_{B(2,2)} W_{m-3,2}^{(2)} = 0$. Hence we have $\tau_{B(3,2)}^{-1} W_{m,3}^{(2)} \notin \EIP(3,2)$ and in view of \cite[2.2.3]{w12}, this implies that $W_{m,3}^{(2)}$ is neither preinjective nor preprojective. Thus $W_{m,3}^{(2)}$ is a regular module with the equal images property and therefore contained in a $\mathbb{Z}A_{\infty}$-component $\mathcal{C}_m$ of $\Gamma(3,2)$.
Due to the fact that $e_0.\tau_{B(3,2)}^{-1}W_{m,3}^{(2)}=0$, the module $\tau_{B(3,2)}^{-1}W_{m,3}^{(2)}$ can not contain a non-trivial module with the equal images property and is hence torsion-free with respect to the torsion pair $(\EIP(3,2), \mathcal{F})$, where $\EKP(3,2) \subseteq \mathcal{F}$, cf. \cite[2.2.3]{w12}. This yields that there is no non-trivial map $W_{m,3}^{(2)} \rightarrow \tau_{B(3,2)}^{-1} W_{m,3}^{(2)}$ and hence $W_{m,3}^{(2)}$ is quasi-simple in $\mathcal{C}_m$. Since $W_{m-3,2}^{(2)} \in \EIP(2,2)$, we obtain $W_{m-3,2}^{(2)} \notin \EKP(2,2)$ by \cite[2.2.3]{w12}, which yields that $\iota_{B(2,2)} W_{m-3,2}^{(2)} \notin \EKP(3,2)$ by \cite[3.32]{w13}. Furthermore, due to the fact that $e_0.\tau_{B(3,2)}^{-1} W_{m,3}^{(2)}=0$, the dual of \cite[2.2.6]{w12} yields that $\tau_{B(3,2)}^{-2} W_{m,3}^{(2)} \in \EKP(3,2)$ and hence $\mathcal{W}(\mathcal{C}_m)=1$. The dual statement holds in view of the duality $\D$ on $\modi B(n,r)$, which is compatible with $\tau$.
\end{proof}
In the proof, we made use of the fact that for $m>3$, we have $\tau_{B(3,2)}^{-1} W_{m,3}^{(2)} \cong \iota_{B(2,2)} W_{m-3,2}^{(2)}$.
Proposition \ref{wmcomps} $(ii)$ now inductively yields for $m>n \geq 3$
$$ \tau_{B(n,2)}^{-1} W_{m,n}^{(2)} \cong \iota_{B(n-1,2)}\cdots \iota_{B(2,2)}W_{m-n,2}^{(2)} $$
and dually
\begin{equation}\label{zwei} \tau_{B(n,2)}M_{m,n}^{(2)} \cong \D \iota_{B(n-1,2)}\cdots \iota_{B(2,2)}W_{m-n,2}^{(2)}. \end{equation}
In particular, we have 
\begin{equation}\label{simple}
\tau_{B(n,2)}M_{n+1,n}^{(2)} \cong S(1).
\end{equation}

An alternative proof for the fact that generalized $W$-modules determine $\mathbb{Z}A_{\infty}$-components in $\Gamma(n,2)$, $n \geq 3$, can be found in \cite[5.7]{w13}.\\
Suppose now that for general $r \geq 2$, we have a $\mathbb{Z}A_{\infty}$-component $\mathcal{C}$ in $\Gamma(n-1,r)$ with $\EKP(n-1,r) \cap \mathcal{C} \neq \emptyset$ which is not completely contained in the category $\EKP(n-1,r)$. It is not known whether, in general, there exist $\mathbb{Z}A_{\infty}$-components that entirely consist of modules with the equal kernels property. At the level of the Kronecker quiver, however, this can not happen \cite[3.3]{w12}. By Theorem \ref{ZA}, the component $\mathcal{C}$ contains an equal kernels cone $(M_{\mathcal{C}} \rightarrow)=\mathcal{C} \cap \EKP(n,r)$ consisting of a distinct quasi-simple module $M=M_{\mathcal{C}} \in \EKP(n,r)$ and all its successors:
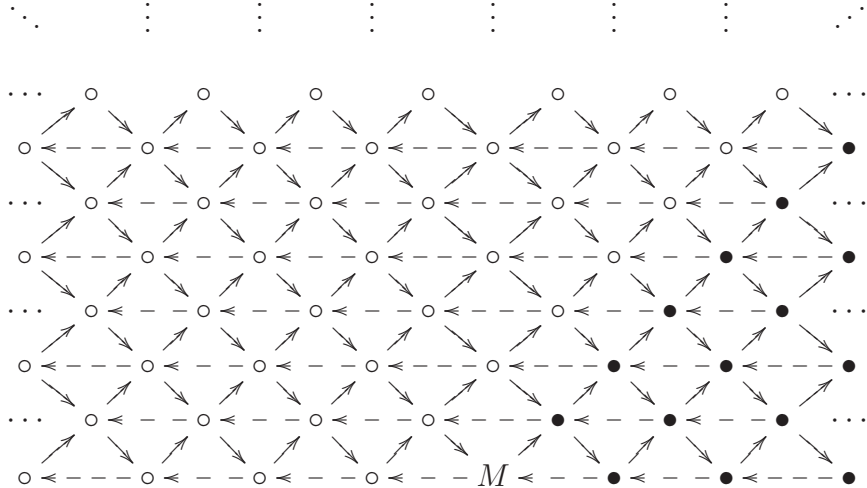
\[\begin{xy}
\xymatrix@R=0.7em@C=0.7em{
\ddots  && \vdots && \vdots &&\vdots  && \vdots && \vdots &&\vdots &&\iddots\\
\cdots &   \circ  \ar[rd]  &&  \circ  \ar[rd]   && \circ   \ar[rd]   &&  \circ   \ar[rd] &&  \circ   \ar[rd]  &&  \circ   \ar[rd] && \circ   \ar[rd]  & \cdots \\
\circ  \ar[rd] \ar[ru] && \circ  \ar[rd] \ar[ru] \ar@{-->}[ll] && \circ  \ar[rd] \ar[ru] \ar@{-->}[ll] && \circ  \ar[rd] \ar[ru] \ar@{-->}[ll] && \circ  \ar[rd] \ar[ru] \ar@{-->}[ll] && \circ  \ar[rd] \ar[ru] \ar@{-->}[ll] && \circ  \ar[rd] \ar[ru]  \ar@{-->}[ll] && \bullet  \ar@{-->}[ll] \\
  \cdots &  \circ  \ar[ru]  \ar[rd]  &&  \circ  \ar[ru]  \ar[rd] \ar@{-->}[ll]  &&  \circ  \ar[ru]  \ar[rd] \ar@{-->}[ll]  &&  \circ  \ar[ru]  \ar[rd] \ar@{-->}[ll] &&  \circ  \ar[ru]  \ar[rd] \ar@{-->}[ll]  &&  \circ  \ar[ru]  \ar[rd] \ar@{-->}[ll]   && \bullet    \ar[rd] \ar@{-->}[ll] \ar[ru] & \cdots \\
\circ  \ar[ru]  \ar[rd] &&   \circ  \ar[ru]  \ar[rd] \ar@{-->}[ll] &&  \circ   \ar[ru]  \ar[rd] \ar@{-->}[ll]  && \circ   \ar[ru]  \ar[rd] \ar@{-->}[ll] &&  \circ  \ar[ru]  \ar[rd] \ar@{-->}[ll] &&  \circ   \ar[ru]  \ar[rd] \ar@{-->}[ll]&&  \bullet  \ar[rd] \ar@{-->}[ll] \ar[ru] &&  \bullet  \ar@{-->}[ll] \\
 \cdots & \circ  \ar[ru]  \ar[rd]  &&   \circ  \ar[ru]  \ar[rd] \ar@{-->}[ll]  &&  \circ  \ar[ru]  \ar[rd] \ar@{-->}[ll]  &&  \circ  \ar[ru]  \ar[rd] \ar@{-->}[ll] &&  \circ  \ar[ru]  \ar[rd] \ar@{-->}[ll]   &&  \bullet  \ar[ru]  \ar[rd] \ar@{-->}[ll] && \bullet  \ar[ru]  \ar[rd] \ar@{-->}[ll] &  \cdots \\
\circ  \ar[rd]  \ar[ru] && \circ  \ar[rd]  \ar[ru] \ar@{-->}[ll] &&  \circ  \ar[rd]  \ar[ru] \ar@{-->}[ll] && \circ  \ar[rd]  \ar[ru] \ar@{-->}[ll] && \circ  \ar[rd]  \ar[ru] \ar@{-->}[ll] &&  \bullet \ar[rd]  \ar[ru] \ar@{-->}[ll] && \bullet \ar[ru] \ar[rd] \ar@{-->}[ll] && \bullet  \ar@{-->}[ll] \\
 \cdots & \circ  \ar[ru]  \ar[rd]  &&  \circ  \ar[ru]  \ar[rd] \ar@{-->}[ll]  &&   \circ  \ar[ru]  \ar[rd] \ar@{-->}[ll]  &&  \circ  \ar[ru]  \ar[rd] \ar@{-->}[ll] && \bullet \ar[ru]  \ar[rd] \ar@{-->}[ll]   && \bullet  \ar[ru]  \ar[rd] \ar@{-->}[ll] && \bullet \ar[rd] \ar[ru]  \ar@{-->}[ll] &  \cdots \\
\circ   \ar[ru] &&  \circ   \ar[ru] \ar@{-->}[ll] &&  \circ   \ar[ru] \ar@{-->}[ll] &&  \circ  \ar[ru] \ar@{-->}[ll] &&  M \ar[ru] \ar@{-->}[ll]&&   \bullet \ar[ru] \ar@{-->}[ll] && \bullet \ar[ru] \ar@{-->}[ll]&& \bullet \ar@{-->}[ll]}
\end{xy} \]
\[\mathcal{C} \hspace{9 cm} \text{ Fig. 2} \] 
\begin{prop}\label{liftcone}
Let $\mathcal{C}$ be a component of $\Gamma(n-1,r)$ as above and let
\begin{equation}\label{n-1}
0 \rightarrow \tau N \rightarrow E \rightarrow N \rightarrow 0
\end{equation}
be an Auslander-Reiten sequence in the subcone $( \tau^{-1}M_{\mathcal{C}} \rightarrow)$ of the equal kernels cone $(M_{\mathcal{C}} \rightarrow)$. Then 
$$0 \rightarrow \iota_{B(n-1,r)} \tau N \rightarrow \iota_{B(n-1,r)} E \rightarrow \iota_{B(n-1,r)} N \rightarrow 0$$
is an Auslander-Reiten sequence in $\Gamma(n,r)$.
\end{prop}
\begin{proof}
In order to determine the lift of (\ref{n-1}) to $\modi B(n,r)$, we need to compute $$\Hom_{B(n-1,r)}(M_{n,n-1}^{(r)}, \tau N)$$ in view of Lemma \ref{lift}. We have $\tau^2 N \in \EKP(n-1,r)$, since $\tau N \in ( \tau^{-1}M_{\mathcal{C}} \rightarrow)$. For $r>2$, we have $\tau M_{n,n-1}^{(r)} \in \EIP(n-1,r)$ according to \cite[2.11]{w12}, whereas in view of (\ref{simple}) we have $\tau M_{n,n-1}^{(2)}\cong S(1) \in \modi B(n-1,2)$. In either case, $\tau M_{n,n-1}^{(r)}$ does not contain a non-trivial factor module with the equal kernels property and in view of \cite[2.2.3]{w12}, we obtain 
$$0 = \Hom_{B(n-1,r)}(\tau M_{n,n-1}^{(r)}, \tau^2 N).$$
The Auslander-Reiten formula yields an isomorphism of vector spaces
\begin{equation}\label{oben}
0 = \overline{\Hom}_{B(n-1,r)}(\tau M_{n,n-1}^{(r)}, \tau^2 N) \cong \underline{\Hom}_{B(n-1,r)}(M_{n,n-1}^{(r)}, \tau N).
\end{equation}
The module $M_{n,n-1}^{(r)}$ is indecomposable non-projective and generated by $e_0.M_{n,n-1}^{(r)}$ while $e_0.P(i)=0$ for $0 < i \leq n-1$ and furthermore $\dim_k e_0.P(0)=1$. Hence we have $\Hom_{B(n-1,r)}(M_{n,n-1}^{(r)},P(i))=0$ for all $0 \leq i \leq n-1$. In view of (\ref{oben}), we thus obtain
$\Hom_{B(n-1,r)} (M_{n,n-1}^{(r)}, \tau N)=0$ and due to Lemma \ref{lift}, the sequence (\ref{n-1}) lifts to the Auslander-Reiten sequence
$$
0 \rightarrow {\begin{pmatrix} \tau N \\ 0 \end{pmatrix}}_0 \rightarrow {\begin{pmatrix} E \\ 0 \end{pmatrix}}_0 \rightarrow {\begin{pmatrix} N \\ 0 \end{pmatrix}}_0 \rightarrow 0
$$
in $\modi B(n,r)$.
\end{proof}

Note that Proposition \ref{wmcomps} implies that the component $\hat{\mathcal{D}}_n$ of $\Gamma(n-1,r)$ that contains the module $M_{n,n-1}^{(r)}$ gives rise to a component  $\mathcal{D}_n$ of $\Gamma(n,r)$  containing the projective module $P(0) \cong M_{n,n}^{(r)} \in \modi B(n,r)$. We now prove our main result:

\begin{theorem}~
\begin{enumerate}[(i)]
\item If $(n,r) \neq (2,2)$, then the module $W_{m,n}^{(r)}$ is quasi-simple in a $\mathbb{Z}A_{\infty}$-component $\mathcal{C}_m$ of $\Gamma(n,r)$ which contains two non-empty disjoint cones $\EIP(n,r) \cap \mathcal{C}_m$ and $\EKP(n,r) \cap \mathcal{C}_m$. Moreover, all modules in $\mathcal{C}_m$ have constant Jordan type.
\item If $r=2$, then $\mathcal{W}(\mathcal{C}_m)=1$ and for $r>2$, we have $\mathcal{W}(\mathcal{C}_m)=0$.
\end{enumerate}
\end{theorem}

\begin{proof}
We consider the dual component $\mathcal{D}_m=\D \mathcal{C}_m$ containing the module $\D W_{m,n}^{(r)}= M_{m,n}^{(r)}$. In view of the duality between $\EIP(n,r)$ and $\EKP(n,r)$, the compatibility between $\tau$ and $\D$ and in view of the fact that the notion of constant Jordan type is self-dual \cite[3.15]{w13}, it suffices to prove the assertion for $\mathcal{D}_m$.

As mentioned in Section \ref{two}, the statement holds for $n=2$ and $r>2$ and due to Proposition \ref{dreizwei} for $(n,r)=(3,2)$. Now let $n \geq n(r)$ with $n(2)=4$ and $n(r)=3$ if $r>2$ and assume that the statement is true for $n-1$. In view of Proposition \ref{wmcomps}, the regular $\mathbb{Z}A_{\infty}$-component $\hat{\mathcal{D}}_m$ of $\Gamma(n-1,r)$ containing the quasi-simple module $M_{\hat{\mathcal{D}}_m}=M_{m,n-1}^{(r)}$ lifts to the component $\mathcal{D}_m$ in which $M_{m,n}^{(r)}$ is quasi-simple and the cone $(\tau^{-1}M_{m,n}^{(r)} \rightarrow)$ coincides via $\iota_{B(n-1,r)}$ with the cone $(\tau^{-1} M_{m,n-1}^{(r)} \rightarrow) \subseteq \hat{\mathcal{D}}_m$ by Proposition \ref{liftcone}.\\
Let $\mathcal{D}=\left\{ \tau^k X \mid X \in (M_{m,n}^{(r)} \rightarrow),\; k \in \mathbb{Z}\right\} \subseteq \mathcal{D}_m$. Due to the fact that $m>n$, $\mathcal{D}$ does not contain $P(0)$ and hence $\mathcal{D}$ does not contain any projective vertices since all modules in $(\tau^{-1} M_{m,n-1}^{(r)} \rightarrow)$ are regular $B(n-1,r)$-modules. Furthermore, $\mathcal{D}$ is $\tau$-stable as well as $\tau^{-1}$-stable and for $X \in \mathcal{D}$, we have $Y \in \mathcal{D}$ if there is an irreducible map $X \rightarrow Y$ or $Y \rightarrow X$. Since $\mathcal{D}_m$ is connected we have $\mathcal{D}_m=\mathcal{D}$ while $\mathcal{D}_m$ is of type $\mathbb{Z}A_{\infty}$ since it is induced by the cone $(\tau^{-1}M_{m,n}^{(r)} \rightarrow)$.\\
According to \cite[2.11]{w12}, we have $\tau M_{m,n}^{(r)} \in \EIP(n,r)$ in case $r>2$. In view of (\ref{zwei}), we have $\tau_{B(n,2)}M_{m,n}^{(2)}\cong \D \iota_{B(n-1,2)}\cdots \iota_{B(2,2)}W_{m-n,2}^{(2)}$. This implies that $\tau_{B(n,2)}M_{m,n}^{(2)} \notin \EIP(n,2)$ due to the fact that $\D W_{m-n,2}^{(2)} \in \EKP(2,2)$. However, since $\D W_{m-n,2}^{(2)} \in \CJT(2,2)$, we obtain that $\D \iota_{B(n-1,2)}\cdots \iota_{B(2,2)}W_{m-n,2}^{(2)} \in \CJT(n,2)$. Moreover, we have $e_{n-1}.\tau_{B(n,2)}M_{m,n}^{(2)}=0$ which implies that $\tau_{B(n,2)}M_{m,n}^{(2)} \notin \EKP(n,2)$ by Remark \ref{equalzero} and that $\tau_{B(n,2)}^2M_{m,n}^{(2)} \in \EIP(n,2)$ by \cite[2.2.6]{w12}. Since in either case all quasi-simple modules in $\mathcal{D}_m$ are of constant Jordan type, we have $\mathcal{D}_m \subset \CJT(n,r)$ by \cite[3.27]{w13}. The foregoing observations yield that $\mathcal{W}(\mathcal{D}_m)=0$ if $r>2$  and $\mathcal{W}(\mathcal{D}_m)=1$ if $r=2$.
\end{proof}

The distribution of equal images and equal kernels modules for $r>2$ is hence as in Fig. 1.

\section*{Acknowledgements}
The results of this paper are part of my doctoral thesis which I have written at the University of Kiel. I thank my advisor Rolf Farnsteiner for his continuous support and for helpful remarks on a preliminary version of this paper.

\addcontentsline{toc}{section}{References}
\bibliography{lit}{}

\begin{thebibliography}{10}

\bibitem{ass06}
I.~Assem, D.~Simson, and A.~Skowro\'{n}ski.
\newblock {\em {Techniques of Representation Theory}}, volume~1 of {\em
  Elements of the Representation Theory of Associative Algebras}.
\newblock Cambridge University Press, Cambridge, 2006.

\bibitem{ars}
M.~Auslander, I.~Reiten, and S.~Smal\o.
\newblock {\em {Representation theory of Artin algebras}}.
\newblock Cambridge University Press, Cambridge, 1995.

\bibitem{bei78}
A.~A. Beilinson.
\newblock {Coherent sheaves on $\mathbb{P}^n$ and problems of linear algebra}.
\newblock {\em Functional Analysis and Its Applications}, 12:214--216, 1978.

\bibitem{cafrpe08}
J.~Carlson, E.~Friedlander, and J.~Pevtsova.
\newblock {Modules of constant Jordan type}.
\newblock {\em Journal f\"ur die Reine und Angewandte Mathematik},
  614:191--234, 2008.

\bibitem{cfs09}
J.~Carlson, E.~Friedlander, and A.~Suslin.
\newblock {Modules for $\mathbb{Z}_p \times \mathbb{Z}_p$}.
\newblock {\em Commentarii Mathematici Helvetici}, 86:609--657, 2011.

\bibitem{erd90}
K.~Erdmann.
\newblock {\em {Blocks of Tame Representation Type and Related Algebras}}.
\newblock Lecture Notes in Mathematics. Springer-Verlag, Cambridge, 1990.

\bibitem{f11}
R.~Farnsteiner.
\newblock {Categories of modules given by varieties of $p$-nilpotent
  operators}.
\newblock Preprint: arXiv 1110.2706.

\bibitem{ker13}
O.~Kerner.
\newblock Private communication, June 2013.

\bibitem{ri78}
C.~M. Ringel.
\newblock {Finite dimensional hereditary algebras of wild representation type}.
\newblock {\em Mathematische Zeitschrift}, 161:235--255, 1978.

\bibitem{ri84}
C.~M. Ringel.
\newblock {\em {Tame algebras and integral quadratic forms}}, volume 1066 of
  {\em Lecture notes in mathematics}.
\newblock Springer Verlag, New York, 1984.

\bibitem{ass07b}
D.~Simson and A.~Skowro\'{n}ski.
\newblock {\em {Representation-Infinite Tilted Algebras}}, volume~3 of {\em
  Elements of the Representation Theory of Associative Algebras}.
\newblock Cambridge University Press, Cambridge, 2007.

\bibitem{w12}
J.~Worch.
\newblock {Categories of modules for elementary abelian $p$-groups and
  generalized Beilinson algebras}.
\newblock {\em Journal of the London Mathematical Society}, 88:649--668, 2013.

\bibitem{w13}
J.~Worch.
\newblock {Module categories and Auslander-Reiten theory for generalized
  Beilinson algebras}.
\newblock http://macau.uni-kiel.de/receive/dissertationdiss00013419, 2013.

\end{thebibliography}
\bibliographystyle{plain} 
\end{document}